\documentclass[11pt]{article}
\usepackage[utf8]{inputenc}
\usepackage{fullpage,amsmath,amssymb,amsthm}
\newcommand{\E}{\mathbb{E}}
\renewcommand{\P}{\mathbb{P}}
\newcommand{\R}{\mathbb{R}}
\newcommand{\EE}{\mathcal{E}}
\newcommand{\LL}{\mathcal{L}}
\newcommand{\var}{\mathrm{Var}}
\newcommand{\C}{\mathbb{C}}
\renewcommand{\H}{\mathbb{H}}
\newcommand{\cmat}{\C^{d\times d}}

\newtheorem{theorem}{Theorem}[section]

\newtheorem{proposition}[theorem]{Proposition}
\theoremstyle{definition}

\begin{document}

\title{Scalar Poincar\'e Implies Matrix Poincar\'e}

\author{
Ankit Garg \\Microsoft Research India \\ \text{garga@microsoft.com} \and Tarun Kathuria \\ EECS, UC Berkeley \\
\text{tarunkathuria@berkeley.edu} \and Nikhil Srivastava\thanks{Supported by NSF Grant CCF-1553751.}\\ Mathematics, UC Berkeley \\ \text{nikhil@math.berkeley.edu}
}

\maketitle

\begin{abstract} 
We prove that every reversible Markov semigroup which satisfies a Poincar\'e inequality satisfies a matrix-valued Poincar\'e inequality for Hermitian $d\times d$ matrix valued functions, with the same Poincar\'e constant. This generalizes recent results \cite{aoun2019matrix,tarun} establishing such inequalities for specific semigroups and consequently yields new matrix concentration inequalities. The short proof follows from the spectral theory of Markov semigroup generators.
\end{abstract}

\section{Introduction}
There is a long tradition in probability theory (see e.g. \cite{gromov1983topological,led}) of using functional inequalities on a probability space $(\Omega,\Sigma,\P)$ to derive concentration inequalities for nice (e.g. Lipschitz) functions $f:\Omega\rightarrow\R$ on that space. The most basic of these inequalities is the Poincar\'e inequality, which postulates that:
\begin{equation}\label{eqn:poincare}\alpha \EE(f,f)\ge \var(f),\end{equation}
for an appropriately large class of $f$, where $\EE(\cdot,\cdot)$ is an appropriate Dirichlet form and $\alpha>0$ is the Poincar\'e constant.

Recently there has been growing interest in extending this phenomenon to { matrix-valued functions} \cite{ch16,CHT15,CH15, aoun2019matrix,tarun}. The last two of these works in particular (independently) studied the notion of {\em matrix Poincar\'e inequality}, in which \eqref{eqn:poincare} is required to hold for $\H_d$-valued $f$, $\EE$, and $\var$, with the inequality replaced by the Lo\"ewner ordering on $\H_d$, the space of $d\times d$ Hermitian matrices. They showed that a matrix Poincar\'e inequality generically implies concentration bounds in the operator norm similar to those in the scalar case\footnote{\cite{aoun2019matrix} also used the related notion of matrix carr\'e du champ operator and obtained  more refined bounds than \cite{tarun} in terms of it.}. They then proceeded to prove matrix Poincar\'e inequalities for several interesting classes of measures (product \cite{aoun2019matrix}, Gaussian \cite{aoun2019matrix}, Strongly Rayleigh \cite{aoun2019matrix, tarun}) on a case by case basis, often mimicking the scalar proofs but requiring significant additional work to handle the noncommutativity of matrices.

In this note, we show that the second step above can also be made generic, and that matrix Poincar\'e inequalities follow automatically from their scalar counterparts in the full generality of arbitrary reversible Markov semigroups (see the excellent book \cite{bakry2013analysis} for a detailed introduction). 

Let $L^2(\Omega,\mu)$ be a separable complex Hilbert space, and let $\C^{d\times d}$ be the Hilbert space of complex $d\times d$ matrices with the Hilbert-Schmidt inner product. To state our theorem, we define a {``matrix-valued inner product''} $\langle\cdot,\cdot\rangle_d$ on the Hilbert space tensor product $L^2(\Omega,\mu)\otimes \cmat \cong L^2(\Omega,\mu;\C^{d\times d})$ as:
\begin{equation} \langle f,g\rangle_d := \int_\Omega f(x)^* g(x)d\mu(x)\in \C^{d\times d},\end{equation}
noting that the integral converges for all $f,g\in L^2(\Omega,\mu)\otimes \cmat$ since
$$\int_\Omega \|f(x)^*g(x)\|_{op} \: d\mu(x)\le \int_\Omega \|f(x)\|_{HS}\|g(x)\|_{HS} \: d\mu(x)\le \|f\|_{L^2(\mu)\otimes\cmat}\|g\|_{L^2(\mu)\otimes\cmat}.$$

\begin{theorem}\label{thm:main} Let $(X_t)_{t\ge 0}$ be a reversible Markov process on a probability space $(\Omega,\Sigma,\P)$ with stationary measure $\mu$ and densely defined self-adjoint infinitesimal generator $\LL:D(\LL)\rightarrow L^2(\Omega,\mu)$. Suppose $\LL$ satisfies a Poincar\'e inequality with constant $\alpha>0$, i.e.,
$$ \alpha \EE (f,f):= \alpha \langle f, -\LL f\rangle \ge \langle f,f\rangle$$
for all $f\in D(\LL)$ with $\E_\mu f=0$. Then
\begin{equation}\label{eqn:dpoinc} \alpha \EE_d(f,f):= \alpha \langle f,(-\LL\otimes I_{\cmat}) f\rangle_d \succeq \langle f,f\rangle_d\end{equation}
for all $f\in D(\LL)\otimes \cmat$ with $\E_\mu f=0$.
\end{theorem}
The domain $D(\LL)$ is always dense in the Dirichlet domain $D(\EE)$ \cite[Section 3.1.4]{bakry2013analysis}, so Theorem 1.1 implies the more conventional form of the inequality for functions in $D(\EE)\otimes \C^{d\times d}$.
Note that for $\H_d$-valued functions $f$, \eqref{eqn:dpoinc} is precisely:
$$ \alpha \int_\Omega f(x) (-\LL f)(x)d\mu(x) \succeq \int_\Omega f(x)^2d\mu(x),$$
which is identical to the matrix Poincar\'e inequality considered in \cite{aoun2019matrix,tarun}. 

Theorem 1.1 implies that any reversible Markov semigroup satisfying a Poincar\'e inequality satisfies an exponential matrix concentration inequality; in particular \cite[Theorem 1.1]{aoun2019matrix} holds with the ``matrix Poincar\'e'' assumption replaced by ``Poincar\'e''. It also allows us to deduce all of the matrix Poincar\'e inequalities derived in \cite{aoun2019matrix,tarun} from their known scalar counterparts, and yields new matrix Poincar\'e and concentration inequalities, notably for Completely Log Concave (i.e., Lorentzian \cite{branden2019lorentzian}) measures via \cite[Theorem 1.1]{anari2019log}.

The proof of Theorem \ref{thm:main} relies on the spectral theorem for unbounded self-adjoint operators on a complex separable Hilbert space. The only property of Markov generators that is used is self-adjointness on an appropriate domain orthogonal to the constant function. Before presenting this proof in Section \ref{sec:unbdd}, we give an elementary linear algebraic proof of the finite-dimensional case in Section \ref{sec:finite}, which is already enough for several important applications (such as all finite reversible Markov chains and strongly log-concave measures) and avoids any analytic subtleties.

\section{Finite Dimensional Case} \label{sec:finite}

Here we prove Theorem \ref{thm:main} when $\Omega$ is finite with $|\Omega| = n$. Let $H = \{f: \Omega \rightarrow \C: \E_\mu f=0\}$. Let $A: H \rightarrow H$ be the operator $A := - \LL$. Consider an orthonormal eigenbasis $g_1,\ldots, g_{n-1}$ of $A$ (with respect to the inner product $\langle f, g\rangle = \sum_{x \in \Omega} \mu(x) \overline{f(x)} g(x)$). Let $\lambda_i$ be the eigenvalue corresponding to $g_i$. By the assumption of Theorem \ref{thm:main}, $\lambda_i \ge 1/\alpha$ for all $i \in [n-1]$. Now consider any $$f =\sum_{i\le n-1} g_i\otimes M_i\in H \otimes \C^{d \times d}.$$ Then
\begin{align*}
    \EE_d(f,f) &= \big \langle f,(A \otimes I_{\cmat}) f \big \rangle_d \\
    &= \Bigg \langle \sum_{i=1}^{n-1} g_i \otimes M_i ,(A \otimes I_{\cmat}) \left(\sum_{j=1}^{n-1} g_j \otimes M_j \right) \Bigg\rangle_d \\
    &= \sum_{i=1}^{n-1} \sum_{j=1}^{n-1} \lambda_j \langle g_i \otimes M_i, g_j \otimes M_j \rangle_d \\
    &= \sum_{i=1}^{n-1} \lambda_i M_i^* M_i \\
    &\succeq \frac{1}{\alpha} \sum_{i=1}^{n-1} M_i^* M_i \\
    &= \frac{1}{\alpha} \sum_{i=1}^{n-1} \sum_{j=1}^{n-1} \langle g_i \otimes M_i, g_j \otimes M_j\rangle_d = \frac{1}{\alpha} \langle f, f \rangle_d.
\end{align*}
In the above calculations, we used the bilinearity of $\langle \cdot, \cdot\rangle_d$ and the fact that $\langle f \otimes M, g \otimes N \rangle_d = \langle f,g\rangle M^* N$ for $f, g \in H$ and $M, N \in \C^{d \times d}$.

\section{Proof of Theorem \ref{thm:main}} \label{sec:unbdd}
Theorem \ref{thm:main} follows from the following proposition by taking $A=-\LL$, $H=L^2(\Omega,\mu)\cap\{f:\E_\mu f =0\}$, and $D(A)=D(\LL)\cap\{f\in H:\E_\mu f=0\}$.

\begin{proposition}
Let $A:D(A)\rightarrow H$ be a densely defined self-adjoint operator on a separable complex Hilbert space $H$ 
satisfying $\langle y,Ay \rangle \ge c\|y\|^2$ for all $y\in D(A)$. Then
	$$ \langle f, (A\otimes I_{\cmat})f\rangle_d \succeq c\langle f,f\rangle_d\quad\forall f\in D(A)\otimes \cmat.$$
\end{proposition}
\begin{proof} For any $v\in \C^d$, define the linear map $(\cdot)_v:H\otimes \cmat\to H\otimes \C^d$ by 
	$$f_v(x) := f(x)v.$$
	Observe that for any $f,g\in H\otimes\cmat$:
	$$v^*\langle f,g\rangle_d v = \int_\Omega v^* f(x)^* g(x) v \: d\mu(x) = \langle f_v,g_v\rangle$$
	where the last inner product on $H\otimes \C^d$ is the standard one (i.e., $\langle f_v,g_v\rangle = \int_\Omega f_v(x)^* g_v(x) d\mu(x)$). Thus, we have for every $f\in D(A)\otimes \cmat$:
	\begin{equation}\label{eqn:fv} v^*\langle f, (A\otimes I_{\cmat})f\rangle_d \: v = \langle f_v, ((A\otimes I_{\cmat}) f)_v\rangle = \langle f_v, (A\otimes I_{\C^d}) f_v\rangle. \end{equation}
	
    Note that $A\otimes I_{\C^d}$ is self-adjoint with domain $D(A)\otimes \C^d\subset H\otimes \C^d$ and $f_v\in D(A)\otimes \C^{d}$.
	Applying the spectral theorem for unbounded operators (e.g., \cite[Theorem VIII.6]{reed1980methods}) and noting that by our assumption the spectrum of $A\otimes I_{\C^d}$ is contained in $[c,\infty)$, we obtain that for some projection valued measure $\{E_\lambda\}_{\lambda\in [c,\infty)}$:
	\begin{equation}\label{eqn:spec}\langle f_v, (A\otimes I_{\C^d}) f_v\rangle = \int_c^\infty \lambda \: d\langle f_v, E_\lambda f_v\rangle \ge c\int_c^\infty d \langle f_v,E_\lambda f_v\rangle = c\|f_v\|^2 = c \: v^* \langle f,f\rangle_d \: v,\end{equation}
	where the integrals are Riemann-Stieltjes integrals.
	Since \eqref{eqn:fv}, \eqref{eqn:spec} hold for every $f\in D(A)\otimes \cmat$ and every $v\in\C^d$, the theorem follows.
\end{proof}
\bibliographystyle{amsalpha}
\bibliography{poincare}

\providecommand{\bysame}{\leavevmode\hbox to3em{\hrulefill}\thinspace}
\providecommand{\MR}{\relax\ifhmode\unskip\space\fi MR }
\providecommand{\MRhref}[2]{%
  \href{http://www.ams.org/mathscinet-getitem?mr=#1}{#2}
}
\providecommand{\href}[2]{#2}
\begin{thebibliography}{ALGV19}

\bibitem[ABY19]{aoun2019matrix}
Richard Aoun, Marwa Banna, and Pierre Youssef, \emph{Matrix {P}oincaré
  inequalities and concentration}, arXiv preprint arXiv:1910.13797, to appear
  in Adv. Math (2019).

\bibitem[ALGV19]{anari2019log}
Nima Anari, Kuikui Liu, Shayan~Oveis Gharan, and Cynthia Vinzant,
  \emph{Log-concave polynomials ii: high-dimensional walks and an fpras for
  counting bases of a matroid}, Proceedings of the 51st Annual ACM SIGACT
  Symposium on Theory of Computing, 2019, pp.~1--12.

\bibitem[BGL13]{bakry2013analysis}
Dominique Bakry, Ivan Gentil, and Michel Ledoux, \emph{Analysis and geometry of
  {M}arkov diffusion operators}, vol. 348, Springer Science \& Business Media,
  2013.

\bibitem[BH19]{branden2019lorentzian}
Petter Br{\"a}nd{\'e}n and June Huh, \emph{Lorentzian polynomials}, arXiv
  preprint arXiv:1902.03719 (2019).

\bibitem[CH16]{ch16}
Hao-Chung Cheng and Min-Hsiu Hsieh, \emph{Characterizations of matrix and
  operator-valued $\phi$-entropies, and operator efron--stein inequalities},
  Proceedings of the Royal Society A: Mathematical, Physical and Engineering
  Sciences \textbf{472} (2016), no.~2187, 20150563.

\bibitem[CH19]{CH15}
\bysame, \emph{Matrix poincar{\'e}, $\phi$-sobolev inequalities, and quantum
  ensembles}, Journal of Mathematical Physics \textbf{60} (2019), no.~3,
  032201.

\bibitem[CHT17]{CHT15}
Hao-Chung Cheng, Min-Hsiu Hsieh, and Marco Tomamichel, \emph{Exponential decay
  of matrix $\phi$-entropies on markov semigroups with applications to
  dynamical evolutions of quantum ensembles}, Journal of Mathematical Physics
  \textbf{58} (2017), no.~9, 092202.

\bibitem[GM83]{gromov1983topological}
Mikhael Gromov and Vitali~D Milman, \emph{A topological application of the
  isoperimetric inequality}, American Journal of Mathematics \textbf{105}
  (1983), no.~4, 843--854.

\bibitem[Kat19]{tarun}
Tarun Kathuria, \emph{A matrix bernstein inequality for strongly rayleigh
  measures}, manuscript (2019).

\bibitem[Led99]{led}
Michel Ledoux, \emph{Concentration of measure and logarithmic sobolev
  inequalities}, Seminaire de probabilites XXXIII, Springer, 1999,
  pp.~120--216.

\bibitem[RS80]{reed1980methods}
Michael Reed and Barry Simon, \emph{Methods of modern mathematical physics.
  volume i: Functional analysis}, Academic press, 1980.

\end{thebibliography}
\end{document}